\documentclass{article}
\usepackage{amsmath,amsthm,amssymb,amsfonts}
\usepackage{verbatim}
\usepackage{graphicx}
\usepackage{stmaryrd} 
\usepackage{hyperref}

\usepackage[colorinlistoftodos]{todonotes}

\theoremstyle{plain}
\newtheorem{thm}{Theorem}
\newtheorem{lem}[thm]{Lemma}
\newtheorem{prop}[thm]{Proposition}

\newtheorem{remark}[thm]{Remark}

\theoremstyle{definition}
\newtheorem{definition}[thm]{Definition}
\newtheorem{exl}[thm]{Example}

\numberwithin{thm}{section}

\newcommand{\adj}{\leftrightarrow}
\newcommand{\adjeq}{\leftrightarroweq}

\def\Z{{\mathbb Z}}
\def\N{{\mathbb N}}
\def\R{{\mathbb R}}

\begin{document}
\title{AFPP and Unions of Convex Disks in the Digital Plane}
\author{Laurence Boxer
\thanks{
    Department of Computer and Information Sciences,
    Niagara University,
    Niagara University, NY 14109, USA;
    and Department of Computer Science and Engineering,
    State University of New York at Buffalo.
    email: boxer@niagara.edu
}
}

\date{ }
\maketitle{}

\begin{abstract}
We use results of~\cite{BxCvexAFPP} to enlarge our knowledge of the
approximate fixed point property (AFPP) for digital images in $\Z^2$.
In particular, we study conditions under which the union of two convex digital
disks has the AFPP.

\end{abstract}


\section{Introduction}
We quote from~\cite{BxCvexAFPP}.
\begin{quote}
The study of fixed points is prominent in many branches of mathematics. In digital topology,
it has become worthwhile to broaden the study to ``approximate fixed points." The Approximate
Fixed Point Property (AFPP), a generalization of the classical fixed point property (FPP),
was introduced in~\cite{BEKLL}.
\end{quote}
At the current writing, there is much to be learned about the AFPP, even for
digital images in the digital plane. In this paper,
we extend the work of~\cite{BxCvexAFPP} in showing how for digital images $X \subset \Z^2$, 
convexity can help us determine whether $(X,c_2)$ has the AFPP.

\section{Preliminaries}
Much of this section is quoted or paraphrased from papers that are listed in the
references, especially~\cite{BxAFPP,BxAFPPtreesProducts,BxConvexity,BEKLL}.

We use $\Z$ to indicate the set of integers; $\R$ for the set of real numbers.

\subsection{Adjacencies}
A digital image is a graph $(X,\kappa)$, where $X$ is a subset of $\Z^n$ for
some positive integer~$n$, and $\kappa$ is an adjacency relation for the points
of~$X$. The $c_u$-adjacencies are commonly used.
Let $x,y \in \Z^n$, $x \neq y$, where we consider these points as $n$-tuples of integers:
\[ x=(x_1,\ldots, x_n),~~~y=(y_1,\ldots,y_n).
\]
Let $u \in \Z$,
$1 \leq u \leq n$. We say $x$ and $y$ are 
{\em $c_u$-adjacent} if
\begin{itemize}
\item There are at most $u$ indices $i$ for which 
      $|x_i - y_i| = 1$.
\item For all indices $j$ such that $|x_j - y_j| \neq 1$ we
      have $x_j=y_j$.
\end{itemize}
Often, a $c_u$-adjacency is denoted by the number of points
adjacent to a given point in $\Z^n$ using this adjacency.
E.g.,
\begin{itemize}
\item In $\Z^1$, $c_1$-adjacency is 2-adjacency.
\item In $\Z^2$, $c_1$-adjacency is 4-adjacency and
      $c_2$-adjacency is 8-adjacency.
\item In $\Z^3$, $c_1$-adjacency is 6-adjacency,
      $c_2$-adjacency is 18-adjacency, and $c_3$-adjacency
      is 26-adjacency.
\end{itemize}

We write $x \adj_{\kappa} x'$, or $x \adj x'$ when $\kappa$ is understood, to indicate
that $x$ and $x'$ are $\kappa$-adjacent. Similarly, we
write $x \adjeq_{\kappa} x'$, or $x \adjeq x'$ when $\kappa$ is understood, to indicate
that $x$ and $x'$ are $\kappa$-adjacent or equal.

Let $(X,\kappa)$ be a digital image and $a,b \in X$. A {\em path from $a$ to $b$ in} $X$ is
a sequence $\{y_i\}_{i=0}^m \subset Y$ such that
$a=y_0$, $b=y_m$, and $y_i \adj_{\kappa} y_{i+1}$ for $0 \leq i < m$.

A subset $Y$ of a digital image $(X,\kappa)$ is
{\em $\kappa$-connected}~\cite{Rosenfeld},
or {\em connected} when $\kappa$
is understood, if for every pair of points $a,b \in Y$ there
exists a $\kappa$-path in $Y$ from $a$ to $b$. A maximal $\kappa$-connected subset of $X$ is a
{\em $\kappa$-component of} $X$.

Given a digital image $(X,\kappa)$ and $x \in X$, we denote by $N^*(X,\kappa,x)$ the set
$\{y \in X \, | \, y \adjeq_{\kappa} x\}$.

\begin{definition}
{\rm \cite{HanNon}}
\label{wedgeDef}
Given digital images $(X_1,\kappa) \subset \Z^n$ and $(X_2, \kappa) \subset \Z^n$, the
image $(X,\kappa)$ is a {\em wedge of $X_1$ and $X_2$}, denoted $X = X_1 \vee X_2$,
if $(X,\kappa)$ is isomorphic to $X_1 \cup X_2$ such that
\begin{itemize}
    \item $X_1 \cap X_2$ is a set with a single point, say, $x_0$, called the {\em wedge point};
          and
    \item if $x,x' \in X$ and $x \adj_{\kappa} x'$, then either $\{x,x'\} \subset X_1$ or
          $\{x,x'\} \subset X_2$.
\end{itemize}
\end{definition}

\subsection{Digitally continuous functions}
The following generalizes a definition of~\cite{Rosenfeld}.

\begin{definition}
\label{continuous}
{\rm ~\cite{Boxer99}}
Let $(X,\kappa)$ and $(Y,\lambda)$ be digital images. A single-valued function
$f: X \rightarrow Y$ is $(\kappa,\lambda)$-continuous if for
every $\kappa$-connected $A \subset X$ we have that
$f(A)$ is a $\lambda$-connected subset of $Y$. $\Box$
\end{definition}

When the adjacency relations are understood, we will simply say that $f$ is \emph{continuous}. Continuity can be expressed in terms of adjacency of points:
\begin{thm}
{\rm ~\cite{Rosenfeld,Boxer99}}
\label{localContinuity}
A function $f:X\to Y$ is continuous if and only if $x \adj x'$ in $X$ implies 
$f(x) \adjeq f(x')$. \qed
\end{thm}

See also~\cite{Chen94,Chen04}, where similar notions are referred to as {\em immersions}, {\em gradually varied operators},
and {\em gradually varied mappings}.

If $f: (X, \kappa) \to (Y, \lambda)$ is a $(\kappa, \lambda)$-continuous bijection such that
$f^{-1}: Y \to X$ is $(\lambda,\kappa)$-continuous, then $f$ is an {\em isomorphism} (called
a {\em homeomorphism} in~\cite{Bx94}), and we say $(X, \kappa)$ and $(Y, \lambda)$ are {\em isomorphic}.

Let $Y \subset X$ and let $r: X \to Y$ be $(\kappa,\kappa)$-continuous such that
$r(y)=y$ for all $y \in Y$. Then $r$ is a {\em $\kappa$-retraction}.

The notation $C(X,\kappa)$ denotes
$\{f: X \to X \, | \, f \mbox{ is }(\kappa,\kappa)-\mbox{continuous}\}$.

For $(x,y) \in \Z^2$, the projection functions $p_1,p_2: \Z^2 \to \Z$ are
\[ p_1(x,y) = x, ~~~ p_2(x,y) = y.
\]

\subsection{Digital convexity, disks}
Material in this section is quoted or paraphrased from~\cite{BxConvexity}.

Let $n>1$. We say a $c_2$-connected 
set $S=\{x_i\}_{i=1}^n \subset \Z^2$ is a
{\em (digital) line segment} if the members of $S$ are collinear. A {\em digital line}
is a $c_2$-connected set $S$ in which all members are collinear and $(S,c_2)$ is isomorphic
to $(\Z, c_1)$.

\begin{remark}
\label{segSlope}
{\rm \cite{BxConvexity}}
A digital line segment or line must be vertical, horizontal, or have
slope of $\pm 1$. We say a segment or line with slope of $\pm 1$ is
{\em slanted}. A horizontal or vertical line is {\em axis parallel}.
\end{remark}

Given a digital line $L \subset \Z^2$, $\Z^2 \setminus L$ has two $c_1$-components. The union of $L$
and a $c_1$-component of $\Z^2 \setminus L$ is a {\em (digital) half-plane}.

A {\em (digital) $\kappa$-closed curve} is a
path $S=\{s_i\}_{i=0}^{m-1}$ such that $i \neq j$ implies $s_i \neq s_j$,
and $s_i \adj_{\kappa} s_{(i+1) \mod \, m}$ for $0 \leq i \leq m-1$.

\begin{definition}
\label{diskDef}
{\rm \cite{BxConvexity}}
Let $S \subset \Z^2$ be a $c_2$-closed curve such that
$\Z^2 \setminus S$ has two $c_1$-components, one finite (possibly empty) and the
other infinite. The union $D$ of $S$ and the finite $c_1$-component 
of $\Z^2 \setminus S$ is a {\em (digital) disk}. $S$ is
a {\em bounding curve} of $D$. The finite $c_1$-component 
of $\Z^2 \setminus S$ is the {\em interior of} $S$, denoted $Int(S)$,
and the infinite $c_1$-component of $\Z^2 \setminus S$ is the {\em exterior of} 
$S$, denoted $Ext(S)$.
\end{definition}

A maximal digital line segment in $S$ is an {\em edge of} $S$.
Note a disk may have multiple distinct bounding curves~\cite{BxConvexity}.

\begin{figure}
    \centering
    \includegraphics[height=1.75in]{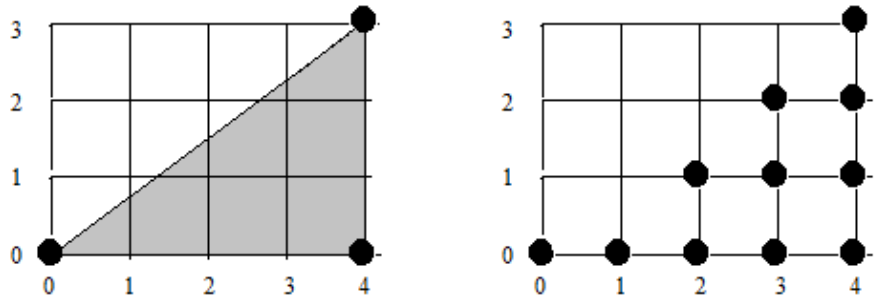}
    \caption{Left: a triangular disc $T$ in the Euclidean plane with vertices $(0,0)$,
    $(4,0)$, and $(4,3)$. Right: the digital image $Y=T \cap \Z^2$. It seems reasonable to
    call $Y \setminus \{(0,0)\}$, but not $Y$, digitally convex. Note also that since the
    only digital line segments among the marked points
    containing $(0,0)$ are (horizontal) subsets of $Y$, it is not sufficient
    to define a digitally convex set as one containing all the digital segments (which must be horizontal, vertical, or
    have slope $\pm 1$) connecting pairs of its points.
    }
    \label{fig:convexRealAndDigital}
\end{figure}

A set $X$ in a Euclidean space $\R^n$ is
{\em convex} if for every pair of distinct
points $x,y \in X$, the line segment
$\overline{xy}$ from $x$ to $y$ is contained in $X$.
The {\em convex hull of} $Y \subset \R^n$,
denoted $hull(Y)$, is the
smallest convex subset of $\R^n$ that contains~$Y$.
If $Y \subset \R^2$ is a finite set, then
$hull(Y)$ is a single point if $Y$ is a singleton;
a line segment if $Y$ has at least 2 members and all are
collinear; otherwise, $hull(Y)$ is a polygonal disk,
and the endpoints of the edges of $hull(Y)$ are its {\em vertices}.

Unsatisfactory attempts to define digital convexity for finite subsets of the digital plane~$\Z^2$ include the following.
\begin{itemize}
    \item We might try defining convexity by the condition that $Y=Y' \cap \Z^2$ where $Y'$ is a convex subset of $\R^2$.
          Figure~\ref{fig:convexRealAndDigital} illustrates why this is unsatisfying;
          the triangular disk $Y' \subset \R^2$ with vertices
          $(0,0)$, $(4,0)$, and $(4,3)$ meets $\Z^2$ in the digital image $Y$ shown on the right side
          of this figure. It appears reasonable to call $Y \setminus \{(0,0)\}$ digitally convex, but not $Y$.
    \item We might try defining convexity by the condition $Y$ contains every digital segment connecting two of its points.
          The example of Figure~\ref{fig:convexRealAndDigital} shows this to be unsatisfying. Note in this digital image,
          the only digital segments containing $(0,0)$ are horizontal and contained in $Y$.
    \item We might try defining convexity by the condition that given $y_0,y_1 \in Y$,
          there is a digital segment in $Y$ from $y_0$ to $y_1$. That this is unsatisfying can be seen in
           Figure~\ref{fig:convexRetract}, where the image shown is one that we want to call convex, although
           its points $(2,2)$ and $(3,4)$ are not connected in $Y$ by a digital segment.
\end{itemize}

We want our definition to capture the feel that a convex disk is a digital version of a Euclidean
convex polygon in which all vertices are integer points and all edges are digital line segments
(hence are axis parallel or slanted).
Thus, we have the following definition.

\begin{definition}
{\rm \cite{BxConvexity}}
A finite set $Y \subset \Z^2$ is 
{\em (digitally) convex} if either
\begin{itemize}
    \item $Y$ is a single point, or
    \item $Y$ is a digital line segment, or
    \item $Y$ is a digital disk with a bounding curve $S$
          such that the endpoints of the edges
          of~$S$ are the vertices of $hull(Y) \subset \R^2$.
\end{itemize}
\end{definition}

\subsection{Approximate fixed points and the AFPP}
\label{approxPrelim}
Let $f \in C(X,\kappa)$
and let $x \in X$. We say
\begin{itemize}
    \item $x$ is a {\em fixed point} of $f$ if $f(x)=x$; 
    \item If $f(x) \adjeq_{\kappa} x$, then
          $x$ is an {\em almost fixed point}~\cite{Rosenfeld,Tsaur} or
          {\em approximate fixed point}~\cite{BEKLL} of 
          $(f,\kappa)$.
    \item A digital image $(X,\kappa)$ has the
          {\em approximate fixed point property} (AFPP)~\cite{BEKLL} if for every $f \in C(X,\kappa)$
          there is an approximate fixed point of $f$. This generalizes the {\em fixed point property}
          (FPP): a digital image $(X,\kappa)$ has the FPP if every $f \in C(X,\kappa)$ has a
          fixed point.
\end{itemize}

The AFPP gathered attention in part because only a digital image with a single point has the
FPP~\cite{BEKLL}.

A. Rosenfeld's paper~\cite{Rosenfeld} states the following as its Theorem~4.1 (quoted verbatim).
\begin{quote}
    Let $I$ be a digital picture, and let $f$ be a continuous function from $I$
    into $I$; then there exists a point $P \in I$ such that $f(P)=P$ or is a neighbor
    or diagonal neighbor of $P$.
\end{quote}
We quote from~\cite{BxAFPP}:
\begin{quote}
Several subsequent papers have incorrectly
concluded that this [Rosenfeld's] result implies that $I$ with
some $c_u$ adjacency has the $AFPP_S$ [in the current paper, the AFPP]. 
By {\em digital picture} Rosenfeld means a digital cube, $I= [0,n]_{\Z}^v$.
By a ``continuous function" he means a $(c_1,c_1)$-continuous function;
by ``a neighbor or diagonal neighbor of $P$" he means a $c_v$-adjacent point.
\end{quote}
Thus, Rosenfeld's result was important but not equivalent to
Theorem~\ref{haveAFPP}(6), below.

\begin{thm}
\label{haveAFPP}
The following digital images have the AFPP.
\begin{enumerate}
    \item Any digital interval $([a,b]_{\Z}, c_1)$~\cite{Rosenfeld}.
    \item Any digital image $(Y,\lambda)$ that is isomorphic to
          $(X,\kappa)$ such that $(X,\kappa)$ has the AFPP~\cite{BEKLL}.
    \item Any digital image $(Y,\kappa)$ that is a retract of
          $(X,\kappa)$ such that $(X,\kappa)$ has the AFPP~\cite{BEKLL}.
    \item Any digital image $(T,\kappa)$ that is a tree~\cite{BxAFPPtreesProducts}.
    \item Any digital image $(X,c_{m+n})$ such that $X = X' \times \Pi_{i=1}^n [a_i,b_i]_{\Z}$,
           $X' \subset \Z^m$, and $(X',c_m)$ has the AFPP~\cite{BxAFPPtreesProducts}.
    \item Any digital cube $(\Pi_{i=1}^n [a_i,b_i]_{\Z}, c_n)$~\cite{BxAFPPtreesProducts}.
    \item Any digitally convex image $(X,c_2) \subset \Z^2$~\cite{BxCvexAFPP}.
    \item The normal product $(X \times Y, NP(\kappa, \lambda))$ of digital images
          $(X, \kappa)$ and $(Y, \lambda)$ that each have the AFPP~\cite{KH21}.
    \item The wedge $(X \vee Y, \kappa)$ of digital images $(X,\kappa)$ and $(Y,\kappa)$ that have the 
          AFPP~\cite{BEKLL}.
\end{enumerate}
\end{thm}

\begin{lem}
\label{ZnSuffices}
Let $r: \Z^n \to X$ be a $c_n$-retraction, where $X$ is finite. Then $(X,c_n)$ has the AFPP.
\end{lem}

\begin{proof}
Since $X$ is finite, there exists $m \in \N$ such that $X \subset Y = [-m,m]_{\Z}^n$. Then
$r|_Y$ is a $c_n$-retraction of $Y$ to $X$. The assertion follows from
Theorem~\ref{haveAFPP}, parts (3) and (6).
\end{proof}

As stated in~\cite{BxCvexAFPP}, ``The next result suggests that `most' digital images 
$(X,c_u) \subset \Z^v$ that have the AFPP have $u=v$." We will therefore
focus our attention on the $c_2$ adjacency in~$\Z^2$.

\begin{thm}
{\rm \cite{BxAFPP}}
\label{c_nMostInteresting}
Let $X \subset \Z^v$ be such that $X$ has a subset $Y = \Pi_{i=1}^v [a_i,b_i]_{\Z}$, where
$v>1$; for all indices $i$, $b_i \in \{a_i,a_i + 1\}$; and, for at least 2 indices $i$,
$b_i=a_i+1$. Then $(X,c_u)$ fails to have the AFPP for $1 \le u < v$.
\end{thm}

\begin{exl}
{\rm \cite{BEKLL}}
\label{SCC4notAFPP}
A digital simple closed curve of at least 4 points does not have the AFPP.
\end{exl}

\section{Retractions to convex images}
In this section, we build retractions that will give us tools for proofs of assertions in
sections~\ref{unionSection} and~\ref{wedgeSection}.

The following proposition can
be useful in determining whether $(X,c_2)$ has the AFPP, for $X \subset \Z^2$. The assertion
following ``Further" was not part of the assertion as stated in~\cite{BxCvexAFPP}, but was
demonstrated in the proof. The version in~\cite{BxCvexAFPP} gave a retraction to $X$
of a digital rectangle that contains $X$, but such a retraction is easily
extended to a retraction of $\Z^2$ that satisfies the asserted properties.

\begin{definition}
Let $X$ be a convex disk in $\Z^2$ with minimal bounding curve $S$.
Let $L_1$ and $L_2$ be distinct parallel digital lines such that for $i \in \{1,2\}$,
$L_i \cap X \subset S$. Then $L_1$ and $L_2$ {\em sandwich} $X$.
\end{definition}

\begin{figure}
    \centering
    \includegraphics[height=3.5in]{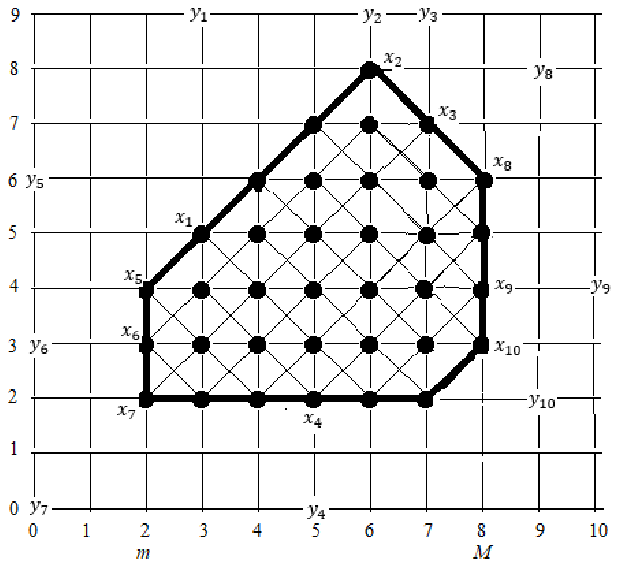}
    \caption{Retraction $r$ of a digital image $Y$ to a subset $X$ that is a convex disk as in
    Theorem~\ref{cvxRetractOfRect}, where for maximally separated vertical (shown here) or
    horizontal lines $L$ and their respective half-planes $H$ that do not contain $X$,
    $r(H)= X \cap L$. \newline
    a) Each point vertically above or below the disk is mapped to its nearest 
       vertical neighbor in $X$, e.g., $r(y_i)=x_i$, $i \in \{1,2,3,4\}$. \newline
    b) Each point to the left (not necessarily horizontally) of $X$ is mapped to 
       the nearest member of $X$ with minimal first coordinate,
       e.g., $r(y_i)=x_i$, $i \in \{5,6,7\}$. \newline
    c) Each point to the right (not necessarily horizontally) of $X$ is mapped to 
       the nearest member of $X$ with maximal first coordinate,
       e.g., $r(y_i)=x_i$, $i \in \{8,9,10\}$. \newline
       Let $L_0$ be the vertical line $x=2$.
       Note $r$ maps the intersection of $Y$ and the half-plane 
       $H_0=\{(x,y) \in \Z^2 \, | \, x \le 2\}$ to the boundary edge
       $\overline{x_5 x_7}=L_0 \cap X$. Similarly, if we take $L_1$ to be the vertical line $x=8$
       and $H_1$ the half-plane of points satisfying $x \ge 8$, we have 
       $r(H_1)=L_1 \cap X = \overline{x_8x_{10}}$.
       }
    \label{fig:convexRetract}
\end{figure}

\begin{prop}
\label{cvxRetractOfRect}
{\rm ~\cite{BxCvexAFPP}}
Let $X \subset \Z^2$, such that $X$ is a
digitally convex disk. Let $S$ be a bounding curve for $X$.
Then there is a $c_2$-retraction $r: \Z^2 \to X$ such that $r(\Z^2 \setminus Int(S)) = S$.
Further, suppose $S$ is a minimal bounding curve for $X$. We can take $r$ to satisfy 
the following.
 Let $L_0$ and $L_1$ be axis parallel lines that sandwich $X$. 
          Let $H_i$ be the half-planes determined by $L_i$, $i \in \{0,1\}$, that
          do not contain $X$. 
          Then $r|_{H_i}$ retracts $H_i$ to $L_i \cap S$, $i \in \{0,1\}$, such that
          $x \in H_i$ implies $r(x)$ is the unique closest (with respect to
          the Euclidean metric) point of $L_i \cap S$ to $x$.
          (See Figure~\ref{fig:convexRetract}).
        
\end{prop}

The importance of convexity in Proposition~\ref{cvxRetractOfRect} and in 
Theorem~\ref{cvxRetractOfRectExtension} below is illustrated in the following.

\begin{figure}
    \centering
    \includegraphics[height=1.75in]{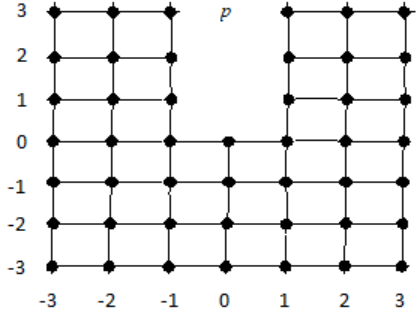}
    \caption{$X$ for Example~\ref{blockUNonRetract} for $n=3$.
       $X$ is not a $c_2$-retract of $\Z^2$, roughly because there is no place in $X$ for a
       $c_2$-retraction to send $p=(0,n) \not \in X$.
       }
    \label{fig:squareAndBlockU}
\end{figure}

\begin{exl}
\label{blockUNonRetract}
Let $X = [-n,n]_{\Z}^2 \setminus (\{0\} \times [1,n]_{\Z})$,  $n > 2$.
(See Figure~\ref{fig:squareAndBlockU}.)
Then $X$ is a non-convex disk that is not a $c_2$-retract of $\Z^2$.
\end{exl}

\begin{proof}
Our argument is modified from the proof of Example 5.11 of~\cite{Bx94}.
Consider the point $p=(0,n) \in \Z^2 \setminus X$. If there were a $c_2$-retraction
$r: \Z^2 \to X$, then by continuity we would have
\[ (-1,n) = r(-1,n) \adjeq_{c_2} r(p) \adjeq_{c_2} r(1,n) = (1,n).
\]
Since there is no point $y \in X$ that is $c_2$-adjacent to both $(-1,n)$
and $(1,n)$, there cannot be such a retraction.
\end{proof}

\begin{figure}
    \centering
    \includegraphics[height=3.5in]{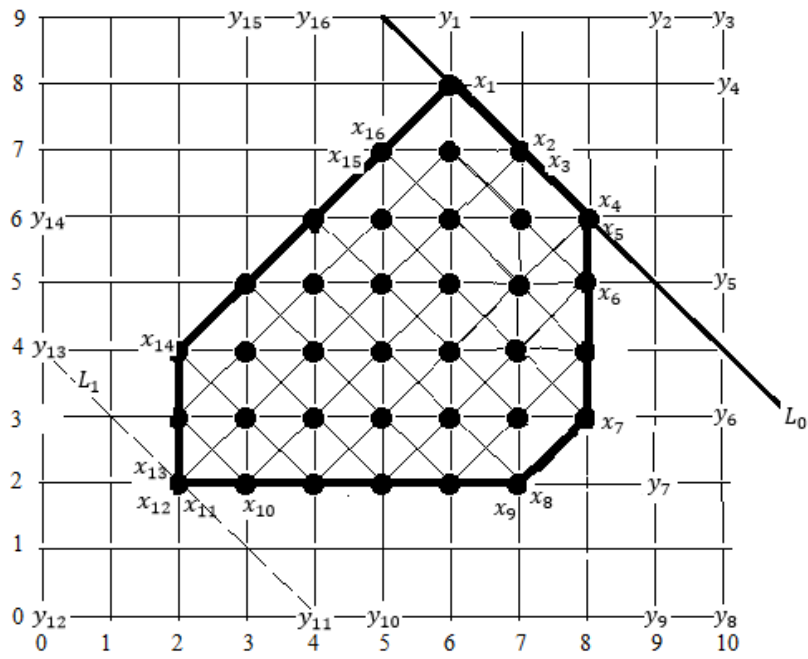}
    \caption{Retraction $r$ of a digital image $Y$ to a subset $X$ that is a convex disk as in
    Theorem~\ref{cvxRetractOfRectExtension}, so that for maximally separated parallel slanted lines 
    $L_0$ and $L_1$ that intersect $X$ and the respective half-planes $H_i$
    determined by $L_i$ not containing $X$ we have $r(Y \cap H_i) = L_i \cap X$.
    We have $r(y_i)=x_i$ for $1 \le i \le 16$.
}
    \label{fig:retractToCnvxDiskSlant}
\end{figure}

We strengthen Proposition~\ref{cvxRetractOfRect} by dropping the requirement that
the lines $L_i$ be axis parallel.

\begin{thm}
\label{cvxRetractOfRectExtension}
Let $X \subset \Z^2$, such that $X$ is a
digitally convex disk. Let $S$ be a bounding curve for $X$.
Then there is a $c_2$-retraction $r: \Z^2 \to X$ such that $r(\Z^2 \setminus Int(S)) = S$.
Further, suppose $S$ is a minimal bounding curve for $X$ and $L_0$ and $L_1$ are distinct 
parallel digital lines in $\Z^2$ that sandwich $X$.
Let $H_i$ be the
half-plane of $\Z^2$ determined by $L_i$ that does not contain $X$. Then we can take $r$
such that $r|_{H_i}$ retracts $H_i$ to $L_i \cap S$, for $i \in \{0,1\}$, such that
          if $x \in L_i$ then $r(x)$ is the unique closest (with respect to
          the Euclidean metric) point of $L_i \cap S$ to $x$.
\end{thm}

\begin{proof}
In light of Proposition~\ref{cvxRetractOfRect}, it suffices to
consider the case that the $L_i$ are slanted. Without loss of generality,
the $L_i$ have slope -1 and therefore 
$P = (x,y) \in L_0$ satisfies $y=-x+b_{\max}$ for
    $b_{\max} = \max \{b \, | (x, -x + b) \in X \}$ and $P = (x,y) \in L_1$ 
    satisfies $y=-x+b_{\min}$ for
     $b_{\min} = \min \{b \, | (x, -x + b) \in X \}$.
     E.g., see Figure~\ref{fig:retractToCnvxDiskSlant}.
\begin{itemize}
\item For $P \in X$, let $r(P)=P$.
\item Suppose $L=L_0$. Let $H_0$ be the half-plane such that
      $P=(x,y) \in H_0 \Leftrightarrow y \ge -x + b_{\max}$.
   (E.g., see Figure~\ref{fig:retractToCnvxDiskSlant}, where $H_0$ is the half-plane above and to 
   the right of the line marked $L_0$.) 
    \begin{itemize}
        \item If $L \cap S$ is a digital segment $\sigma$ and $P \in H_0$ can drop a perpendicular to $\sigma$ at a point $Q \in \sigma$, then
          $r(P)=Q$ (e.g., $P \in \{y_2,y_4\}$ in Figure~\ref{fig:retractToCnvxDiskSlant}).
        \item If $L \cap S$ is a digital segment $\sigma$ and $P$ can drop
              a perpendicular to $L$ at a point between the endpoints
           of $\sigma$ that does not meet a point of $\sigma$, then $r(P) = r(x-1,y)$ (e.g.,
           $r(y_3)=r(y_2)$ in Figure~\ref{fig:retractToCnvxDiskSlant}).
        \item If $L \cap S$ is a digital segment $\sigma$ and a perpendicular from $P \in H_0$ to $L$ does not fall between the endpoints of $\sigma$,
           there is a unique nearest (in the Euclidean metric) endpoint $Q$ of $\sigma$ to $P$.
           Then $r(P)=Q$ (e.g., $P \in \{y_1, y_5\}$ in Figure~\ref{fig:retractToCnvxDiskSlant}).
        \item If $L \cap S$ is a single point $Q$ then $r(P)=Q$ for all $P \in H$.
    \end{itemize}
     \item Suppose $L = L_1$. Let $H_1$ be the half-plane bounded by $L$ not containing $X$.
   (In Figure~\ref{fig:retractToCnvxDiskSlant}, $L_1$ is the line containing $\{y_{11}, y_{13}\}$,
    and $H_1$ is the half-plane below and to the left of $L_1$.)
        Then $P \in H_1$ implies $r(P) \in \sigma$ is defined in a fashion similar to that of the case $L = L_0$.
    \item For $P \in \Z^2 \setminus (X \cup H_0 \cup H_1)$, if a parallel to $L_0$ (and $L_1$)
        through $P$ meets a 
       nearest $Q \in S$ then $r(P)=Q$ (e.g., $P \in \{y_6,y_7,y_8,y_{13},y_{14}, y_{15}\}$
       in Figure~\ref{fig:retractToCnvxDiskSlant}). Otherwise, for $P=(u,v)$,
       $r(P)=r(u-1,v)$ (e.g., in Figure~\ref{fig:retractToCnvxDiskSlant}, $r(y_8)=r(y_9)$ and
       $r(y_{16})=r(y_{15})$).
\end{itemize}

We show that $r \in C(\Z^2,c_2)$ in the following; it will follow easily that $r$ is a retraction with the asserted
properties.

Since $r|_X$ is an inclusion function, $r$ satisfies the continuity condition 
of Theorem~\ref{localContinuity} for all $P \in X \setminus S$.

Next, consider regions of $\Z^2 \setminus X$. Let $P=(x,y)$. Note the $c_2$-adjacent points
          to $P$ are $(x+i,y+j)$, where $i,j \in \{-1,0,1\}$ and $(i,j) \neq (0,0)$.

Consider $r$ for $P \in H_0 \setminus S$.
   \begin{itemize}
    \item If $L_0 \cap X$ is a digital segment and $P$ has a perpendicular to  $L_0$ at
          $Q = (u,v)$, then for each $P'$ that is $c_2$-adjacent to $P$, 
          \[ r(P') \in \{(u-1,v+1), Q, (u+1,v-1)\}.
          \]
          Thus $r(P') \adjeq_{c_2} r(P)$. (See, e.g.,
          Figure~\ref{fig:retractToCnvxDiskSlant}, with $P=y_2$, $r(y_3) = r(y_2) \adj r(y_4)$.) 
          Thus $r$ is continuous at $P$.
    \item If $L_0 \cap X$ is a digital segment and $P$ has a perpendicular to the (real) line
          containing $L_0$ that does not meet a point of $L_0$, then $r(P)=r(x-1,y)$. Suppose $r(x-1,y)=(u,v)$. Then for
          each $P'$ that is $c_2$-adjacent to $P$, $r(P') \in \{(u,v), (u+1,v-1)\}$. (See, e.g.,
          Figure~\ref{fig:retractToCnvxDiskSlant}, with $P=y_3$ and $r(y_3) = r(y_2) \adj r(y_4)$.)
          Thus $r$ is continuous at $P$.
    \item If $L_0 \cap X$ is a digital segment from $(a,b)$ to $(c,d)$ and the perpendicular 
          projection of $P$ to the line of $L_0$ is $(u,v)$
          such that $u > c > a$, then $r(P)= (c,d)$. For each $P'$ that is 
          $c_2$-adjacent to $P$, $r(P') =(c,d)$. (See, e.g.,
          Figure~\ref{fig:retractToCnvxDiskSlant}, with $P=y_5$.)
          Thus $r$ is continuous at $P$.
    \item If $L_0 \cap X$ is a digital segment from $(a,b)$ to $(c,d)$ and the perpendicular 
          projection of $P$ to the line of $L_0$ is $(u,v)$
          such that $u < a < c$, then $r(P)= (a,b)$. If $P' \adj_{c_2} P$, then
          $r(P') = (a,b)$. Thus $r$ is continuous at $P$. (See, e.g.,
          Figure~\ref{fig:retractToCnvxDiskSlant}, with $P=y_1$).
    \item If $L_0 \cap X$ is a single point $(a,b)$ and $P'$ is
          $c_2$-adjacent to $P$, then $r(P)=(a,b)=r(P')$. Hence $r$ is continuous at $P$.
\end{itemize}

Thus $r$ is continuous at every $P \in H_0 \setminus S$.

Similarly, $r$ is continuous at every $P \in H_1 \setminus S$.

Similarly, $r$ is continuous at every $P \in \Z^2 \setminus (X \cup H_0 \cup H_1)$.

For $P' \adj_{c_2} P$ where $P \in S$, we have the following cases.
\begin{itemize}
    \item If $P' \in X$ then $r(P')=P' \adj P = r(P)$.
    \item If $P' \not \in X$, then we have seen above that $r$ is continuous at $P'$,
          so $r(P) \adjeq r(P')$.
\end{itemize}
Hence $r$ is continuous at $P$. 

Thus $r \in C(X,c_2)$. 

It follows from the above that $r|_{H_i}$ retracts $H_i$ to $L_i \cap S$, $i \in \{0,1\}$, such that
          if $x \in L_i \setminus S$ then $r(x)$ is the unique closest point of $L_i \cap S$ to $x$.
This completes the proof.
\end{proof}

\section{Union of convex images meeting in a common edge}
\label{unionSection}
In this and the following sections, we extend our knowledge of digital images 
$(X,c_2) \subset \Z^2$ that have the AFPP.
Most of our results are derived by showing the applicability of 
Lemma~\ref{ZnSuffices}.

In several of the assertions in this section, we will be concerned with the relation between a digital
disk $X$ and a half-plane $H$. The reader is cautioned that convenience sometimes dictates $X \subset H$
and sometimes $X$ is on the other side of the boundary of $H$;
and sometimes $X \cap H$ is an edge of $X$; and sometimes
$X \cap H$ is an endpoint of an edge of $X$.

\begin{definition}
Let $X_1$ and $X_2$ be subsets of $\Z^2$. Let $X = X_1 \cup X_2$.
Let $L$ be a digital line in $\Z^2$ such that $X_1$ and $X_2$ are on opposite sides of $L$,
and either
\begin{itemize}
    \item $X_1 \cap X_2$ is an edge of both $X_1$ and $X_2$, or
    \item $(X,c_2) = (X_1 \vee X_2, c_2)$ with wedge point $x_0$ such that $X \cap L = \{x_0\}$.
\end{itemize}
Then $L$ is a {\em line of separation of} $X_1$ and $X_2$.
\end{definition}

\begin{figure}
    \centering
    \includegraphics[height=1.5in]{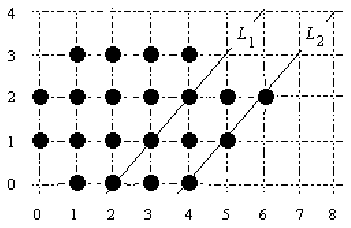}
    \caption{A digital image $X$ to illustrate Theorem~\ref{maxSeg}. \newline
   $L$ (line of separation of $X_1$ and $X_2$):~ $y=x-2$ \newline
    $X_1 = \{(x,y) \in X ~ | ~ y \ge x-2\}$ ~~~~~ $X_2 = \{(x,y) \in X ~ | ~ y \le x-2 \}$ \newline 
$H_1$ and $H_2$, are half-planes respectively Northwest of $L$ and Southeast of $L$.
}
    \label{fig:separates}
\end{figure}

\begin{thm}
\label{maxSeg}
Let $X_1 \subset \Z^2$, $X_2 \subset \Z^2$ be digitally convex disks such that
$\sigma = X_1 \cap X_2$ is an edge of both $S_1$ and
$S_2$, where $S_i$ is a minimal bounding curve of $X_i$. Then
\begin{itemize}
    \item there is a $c_2$-retraction $r: \Z^2 \to X=X_1 \cup X_2$; 
          and
    \item $(X,c_2)$ has the AFPP.
\end{itemize}
\end{thm}

\begin{proof}
(See Figure~\ref{fig:separates}.)
Let $L$ be a line of separation of $X_1$ and $X_2$. 
We have $\sigma \subset L$. 
Let $H_i$ be the half-plane determined by $L$ containing $X_i$.
By Theorem~\ref{cvxRetractOfRectExtension}, there exist retractions $r_i: \Z^2 \to X_i$
such that $r_1(H_2)=r_2(H_1)=\sigma$, and
$r_i$ retracts $L_1$ to $\sigma$ for $i \in \{1,2\}$ such that 
\begin{equation}
\label{2bRetract}
          \mbox{$x \in L \setminus \sigma$ implies $r_i(x)$ is the unique nearest point of $\sigma$ to $x$}.
\end{equation}
The function $r: \Z^2 \to X$ defined by
$r(p)= r_i(p)$ for $p \in H_i$,
is, by~(\ref{2bRetract}), well defined and $c_2$-continuous. Thus $r$ is a retraction of $\Z^2$ to $X$.
It follows from Lemma~\ref{ZnSuffices} that $X$ has the AFPP.
\end{proof}

\begin{remark}
\label{notMaximal}
Theorem~\ref{maxSeg} does not cover the case in which $\sigma = X_1 \cap X_2 \subset L$ and
$\sigma$ is not a maximal segment of both $S_1$ and $S_2$, as in Figure~\ref{fig:tee}. In fact, the
argument given for Theorem~\ref{maxSeg} does not generally work in such a case; see Example~\ref{teeExl}.
\end{remark}

\begin{figure}
    \centering
    \includegraphics[height=2in]{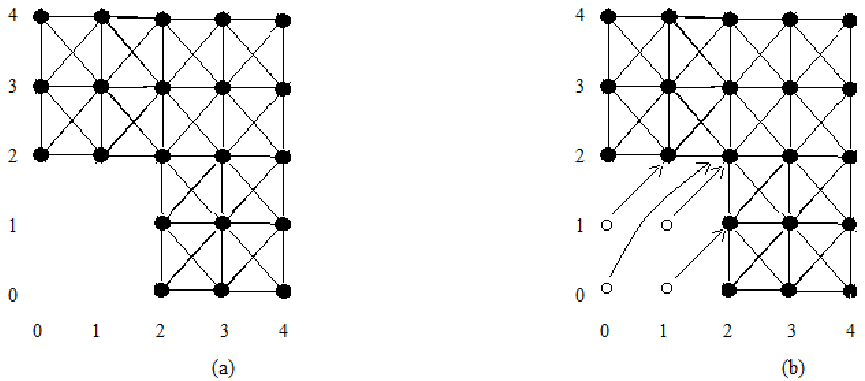}
    \caption{(a) An image $X$ illustrating Remark~\ref{notMaximal} and Example~\ref{teeExl}. Let $X=X_1 \cup X_2$, where
    $X_1 = [0,4]_{\Z} \times [2,4]_{\Z}$, $X_2 = [2,4]_{\Z} \times [0,2]_{\Z}$, $L$ is the
    line $y=2$. We have $X_1 \cap X_2 = \{(2,2), (3,2), (4,2)\}$ is not a maximal segment
    of the bounding curve of $X_1$.\newline
    (b) Arrows show $R(P)$ for $P \not \in X$, for a $c_2$-retraction 
    $R: [0,4]_{\Z} \times [0,4]_{\Z} \to X$.
       }
    \label{fig:tee}
\end{figure}

\begin{exl}
\label{teeExl}
For $X=X_1 \cup X_2$, where
\[ X_1 = [0,4]_{\Z} \times [2,4]_{\Z}, ~~~ X_2 = [2,4]_{\Z} \times [0,2]_{\Z}
\]
(see Figure~\ref{fig:tee}(a)), $r$ of Theorem~\ref{maxSeg} is not well defined at the 
point $(0,2)$, since $r_1(0,2) = (0,2) \in X_1 \setminus X_2$ and $r_2(0,2) \in X_2$. 
However, we can show $(X,c_2)$
has the AFPP as follows. Define $R: [0,4]_{\Z} \times [0,4]_{\Z} \to X$ by
\[ R(P) = \left \{ \begin{array}{ll}
          P & \mbox{if } P \in X; \\
          (1,2) & \mbox{if } P = (0,1); \\
          (2,2) & \mbox{if } P \in \{(0,0), (1,1)\}; \\
          (2,1) & \mbox{if } P = (1,0) \\
         \end{array} \right .
\]
(see  Figure~\ref{fig:tee}(b)).
It is easily seen that $R$ is a $c_2$-retraction of $[0,4]_{\Z} \times [0,4]_{\Z}$
to $X$, so the assertion follows from Theorem~\ref{haveAFPP}, assertions (3) and (6).
\end{exl}

One might ask if Theorem~\ref{maxSeg} extends to unions of more than 2 convex disks $X_i$, such that
$\sigma_{i-1}=X_{i-1} \cap X_i$ is a maximal segment of both $S_{i-1}$ and $S_i$, where $S_i$
is a minimal bounding curve of $X_i$. The following shows that such an
extension is not generally valid.

\begin{figure}
    \centering
    \includegraphics[height=3in]{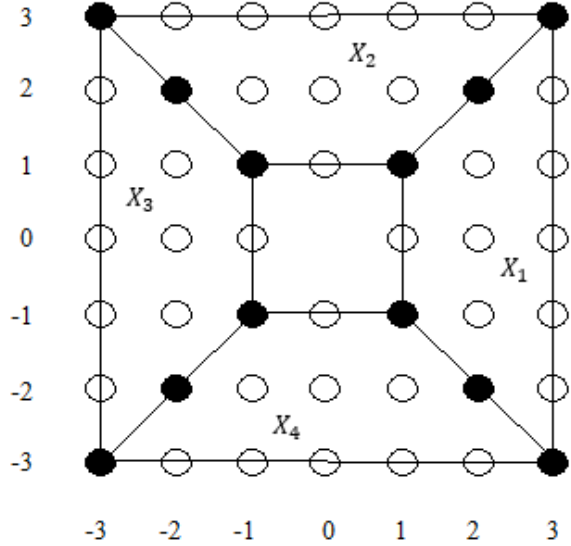}
    \caption{The digital image $X$ of Example~\ref{annulusExl}. Most adjacencies are not
     shown in order to clarify the $X_i$. Points $(x,y) \in X$ such that $|x| = |y|$ are shown in
     solid circles; these are points at which two of the $X_i$ intersect.
       }
    \label{fig:annulus3}
\end{figure}

\begin{exl}
\label{annulusExl}
Let $X = [-3,3]_{\Z}^2 \setminus \{(0,0)\} = \bigcup_{i=1}^4 X_i$, where
\[ X_1 = \{(x,y) \in X~|~ 1 \le x \le 3, -x \le y \le x \}, \]
\[   X_2 = \{(x,y) \in X~|~ 1 \le y \le 3, ~-y \le x \le y\}, \]
\[ X_3 = \{(x,y) \in X~|~ -3 \le x \le -1, ~x \le y \le -x \}, \]
\[   X_4 = \{(x,y) \in X~|~-3 \le y \le -1, ~y \le x \le -y \}. \]
See Figure~\ref{fig:annulus3}.
Then each $X_i$ is a convex disk, and each pair $(X_1, X_2)$, $(X_2, X_3)$, $(X_3,X_4)$,
$(X_4,X_1)$ intersects in a common edge of bounding curves for both members of the pair. Then
$(X,c_2)$ does not have the AFPP.
\end{exl}

\begin{proof}
Let $U = \{(x,y) \in X~|~ |x| = 1 \mbox{ or } |y| = 1 \}$. ($U$ is the ``inner ring" of $X$.)
Let $r: X \to U$ be defined for $P=(x,y)$ by
\[ r(P) = \left \{ \begin{array}{ll}
    (1,-1) & \mbox{ if } P \in X_1, ~y \le -1; \mbox{ or if } P \in X_4, ~x \ge 1;\\
    (1,y) & \mbox{ if } P \in X_1, ~-1 \le y \le 1; \\
    (1,1) & \mbox{ if } P \in X_1, ~y \ge 1; \mbox{ or if } P \in X_2, ~x \ge 1; \\
    (x,1) & \mbox{ if } P \in X_2, ~-y \le x \le y; \\
    (-1,1) & \mbox{ if } P \in X_2, ~x \le -1; \mbox{ or if } P \in X_3, ~y \ge 1;\\
    (-1,y) & \mbox{ if } P \in X_3, ~-1 \le y \le 1; \\
    (-1,-1) & \mbox{ if } P \in X_3, ~y \le -1; \mbox{ or if } P \in X_4, ~x \le -1; \\
    (x,-1) & \mbox{ if } P \in X_4, ~-1 \le x \le 1.
\end{array} \right .
\]

It is easy to see that $r$ is well-defined and $c_2$-continuous, hence is a $c_2$-retraction
of $X$ to $U$. By Example~\ref{SCC4notAFPP}, we know $(U,c_2)$ does not have the AFPP. It follows from
Theorem~\ref{haveAFPP}(3) that $(X,c_2)$ does not have the AFPP.
\end{proof} 

\section{Wedges of convex images}
\label{wedgeSection}
In this section, we obtain a result somewhat similar to Theorem~\ref{maxSeg}
by showing that the wedge of digital images in $(\Z^2, c_2)$ has the AFPP.
 
\begin{lem}
\label{wedgePtAtEndpt}
Let $(X,c_2) = (X_1,c_2) \vee (X_2,c_2) \subset \Z^2$, where $X_1$ and $X_2$ are convex disks.
Let $x_0$ be the wedge point. Then $x_0$ must be an endpoint of edges of both $X_1$ and $X_2$.
\end{lem}

\begin{proof}
Suppose $X_1 \cap X_2 = \{x_0\}$ where $x_0$ is in a boundary edge $\sigma$ of, say, $X_1$ but is not
an endpoint of $\sigma$. Then, whether the interior angle of the convex disk $X_2$ at $x_0$ measures
$45^{\circ}$ ($\pi /4$ radians), $90^{\circ}$ ($\pi /2$ radians), or $135^{\circ}$ ($3 \pi /4$ radians),
there must be points $x_i \in X_i$ such that $x_1 \neq x_0 \neq x_2$ and
$x_1$, $x_0$, and $x_2$ are pairwise $c_2$-adjacent. Therefore, 
$(X_1 \cup X_2,c_2) \neq (X_1,c_2) \vee (X_2,c_2)$. This is contrary to hypothesis, so the assertion is established.
\end{proof}

\begin{prop}
\label{separation}
Let $X \subset \Z^2$ be such that $(X,c_2) = (X_1 \vee X_2, c_2)$, where $X_1$ and $X_2$ are convex disks.
Then there is a digital line $L \subset \Z^2$ that is a line of separation of $X_1$ and $X_2$.
\end{prop}

\begin{proof}
Since the edges of a convex disk must be horizontal, vertical, or slanted (having slopes of $\pm 1$),
interior angles formed by edges of the disk must measure
$45^{\circ}$ ($\pi / 4$ radians), $90^{\circ}$ ($\pi / 2$ radians), or $135^{\circ}$ ($3 \pi / 4$ radians). 
Let $(X,c_2) = (X_1,c_2) \vee (X_2,c_2) \subset \Z^2$ where $X_1$ and $X_2$ are convex disks, and
$x_0$ is the wedge point. 

\begin{figure}
    \centering
    \includegraphics[height=1in]{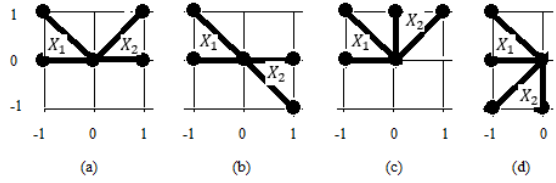}
    \caption{Up to isomorphism, all possibilities are shown for $N^*(X,c_2,(0,0))$ for $X = X_1 \cup X_2$,
     where $X_1$ and $X_2$ are convex disks, $X_1 \cap X_2 = \{(0,0)\}$,
     and both interior angles are $45^{\circ}$ ($\pi / 4)$ radians). Some adjacency lines not shown.
     Subfigures (a), (d), and (e) are not appropriate to  $(X,c_2)$ as $(X_1 \vee X_2, c_2)$, since
     each has points $p \in X_1$, $q \in X_2$ such that $p \neq (0,0) \neq q$ and $p \adj_{c_2} q$.
    Subfigures (b) and (c) are appropriate to  $(X,c_2)$ as $(X_1 \vee X_2, c_2)$.
    (Not meant to be understood as all of $X$.)
       }
    \label{fig:45-45}
\end{figure}

Figure~\ref{fig:45-45}, subfigures (b) and (c), show up to isomorphism the only ways in which
$X_1$ and $X_2$ can meet at a point $x_0$ at which both have 
$45^{\circ}$ ($\pi / 4$ radians) interior angles
in $X_1 \vee X_2$. Clearly these configurations permit a line of separation $L$.
Subfigures (a), (d), and (e) show other ways in which 
$X_1$ and $X_2$ can have a one-point intersection at which both have
$45^{\circ}$ ($\pi / 4$ radians) interior angles, but not
in $X_1 \vee X_2$, since there exist $x_1 \in X_1$, $x_2 \in X_2$ such that $x_1 \adj x_2$ and
$x_1 \neq x_0 \neq x_2$.

\begin{figure}
    \centering
    \includegraphics[height=1in]{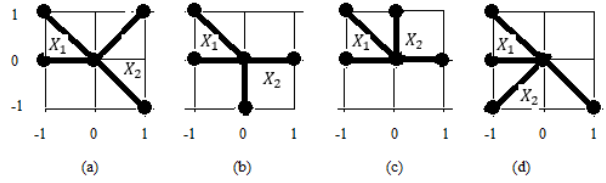}
    \caption{Up to isomorphism, all possibilities are shown for $N^*(X,c_2,(0,0))$ for $X = X_1 \cup X_2$,
     where $X_1$ and $X_2$ are convex disks, $X_1 \cap X_2 = \{(0,0)\}$,
     and the interior angles are $45^{\circ}$ ($\pi / 4)$ radians) and $90^{\circ}$ ($\pi / 2$ radians).
     Some adjacency lines not shown.
     Subfigures (a), (c), and (d) are not appropriate to $(X,c_2)$ as $(X_1 \vee X_2, c_2)$
     , since each has points $p \in X_1$, $q \in X_2$ such that $p \neq (0,0) \neq q$
     and $p \adj_{c_2} q$. Subfigure (b) is appropriate to $(X,c_2)$ as $(X_1 \vee X_2, c_2)$.
     (Not meant to be understood as all of $X$.)
       }
    \label{fig:45-90}
\end{figure}

Figure~\ref{fig:45-90}(b) shows up to isomorphism the only way in which
$X_1$ and $X_2$ can meet in a single point (at $x_0$) at which one has a $45^{\circ}$ ($\pi / 4$ radians) 
interior angle and the other has a $90^{\circ}$ ($\pi / 2$ radians) interior angle 
in $X_1 \vee X_2$. Clearly this configuration permits a line of separation $L$.
Subfigures (a), (c), and (d) show other ways in which 
$X_1$ and $X_2$ can meet at a single point at which one has a $45^{\circ}$ ($\pi / 4$ radians) 
interior angle and the other has a $90^{\circ}$ ($\pi / 2$ radians) interior angle, but not
in $X_1 \vee X_2$, since in each of these configurations
there exist $x_1 \in X_1$, $x_2 \in X_2$ such that $x_1 \adj x_2$ and
$x_1 \neq x_0 \neq x_2$.

\begin{figure}
    \centering
    \includegraphics[height=1in]{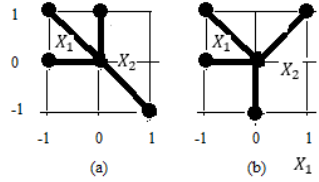}
    \caption{Up to isomorphism, all possibilities are shown for $N^*(X,c_2,(0,0))$ for $X = X_1 \cup X_2$,
     where $X_1$ and $X_2$ are convex disks, $X_1 \cap X_2 = \{(0,0)\}$,
     and the interior angles are $45^{\circ}$ ($\pi / 4$ radians) and $135^{\circ}$ ($3 \pi / 4$ radians).
     Some adjacency lines not shown.  None of these is suitable for $(X,c_2)$ as $(X_1 \vee X_2, c_2)$, as each
     has $p \in X_1$ and $q \in X_2$  such that $p \neq (0,0) \neq q$ and $p \adj_{c_2} q$.
     (Not meant to be understood as all of $X$.)
       }
    \label{fig:45-135}
\end{figure}

Figure~\ref{fig:45-135} illustrates that there is no way
for $X_1$ and $X_2$ to meet at a wedge point for $X_1 \vee X_2$ at which one has
a $45^{\circ}$ ($\pi / 4$ radians) interior angle and the other has a $135^{\circ}$ ($3 \pi / 4$ radians) interior angle.

Figure~\ref{fig:90-90}, subfigures (a) and (b), show up to isomorphism the only ways in which
$X_1$ and $X_2$ can meet in a single point (at $x_0$) at which both have a 
$90^{\circ}$ ($\pi / 2$ radians) interior angle in $X_1 \vee X_2$. 
Clearly these configurations permit a line of separation $L$.

It is easily seen that in $X_1 \vee X_2$ we cannot have, at the wedge point $x_0$, 
one of $X_1$ and $X_2$ with an interior angle
of $90^{\circ}$ ($\pi / 2$ radians) and the other with an interior angle of $135^{\circ}$ ($3 \pi / 4$ radians);
nor can we have both of $X_1$ and $X_2$ with interior angles of $135^{\circ}$ ($3 \pi / 4$ radians). The assertion follows.
\end{proof}

\begin{figure}
    \centering
    \includegraphics[height=1in]{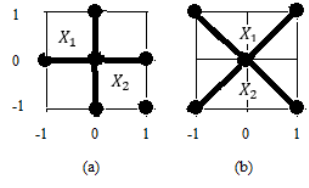}
    \caption{Up to isomorphism, the only possibilities are shown for $N^*(X,c_2,(0,0))$ for $X = X_1 \cup X_2$,
     where $X_1$ and $X_2$ are convex disks, $X_1 \cap X_2 = \{(0,0)\}$,
     and the interior angles are both $90^{\circ}$ ($\pi / 2$ radians).
     Some adjacency lines not shown.  (a) is not suitable for $(X,c_2)$ as $(X_1 \vee X_2, c_2)$, as, e.g.,
     $p=(0,1) \in X_1$ and $q=(1,0) \in X_2$ are such that $p \neq (0,0) \neq q$ and $p \adj_{c_2} q$.
     (b) is suitable for $(X,c_2) = (X_1 \vee X_2, c_2)$.
    (Not meant to be understood as all of $X$.)
           }
    \label{fig:90-90}
\end{figure}

\begin{thm}
\label{wedgeRetract}
Let $X_1, X_2  \subset \Z^2$ be digitally convex disks. \newline 
Let $(X,c_2) = (X_1 \vee X_2,c_2)$. Then
\begin{itemize}
    \item there is a $c_2$-retraction $r: \Z^2 \to X$; and
    \item $(X,c_2)$ has the AFPP.
\end{itemize}
\end{thm}

\begin{proof}
Let $x_0$ be the wedge point of $X$. By Lemma~\ref{wedgePtAtEndpt}, $x_0$ is an endpoint of an edge of
$X_1$ and of an edge of $X_2$. Let $S_2$ be a minimal bounding curve for $X_2$.

By Proposition~\ref{separation}, there is a digital line $L_1 \subset \Z^2$
that is a line of separation of $X_1$ and $X_2$.
By Theorem~\ref{cvxRetractOfRectExtension}, there exist retractions $r_i: \Z^2 \to X_i$ such that
\begin{equation}
\label{comeTogether}
    r_1(H_2) = \{x_0\}=r_2(H_1),
\end{equation}
where $H_i$ is the half-plane determined by $L_1$ containing $X_i$, and
$r_2|_H$ retracts $H$ to $L_2 \cap X_2$, where $H$ is the half-plane determined by $L_2$ not containing $X_2$, such that
\begin{equation}
\label{mapClose}
    \mbox{if $x \in L_i \setminus S$ then $r_i(x)$ is the unique closest point of $L_i \cap S$ to $x$.}
\end{equation}
The function $r: \Z^2 \to X$ given by
\[ r(y) = \left \{ \begin{array}{ll}
   r_1(y)  & \mbox{if } y \in H_1; \\
   r_2(y)  & \mbox{if } y \in H_2,
\end{array} \right .
\]
is, by~(\ref{comeTogether}) and~(\ref{mapClose}), well-defined and $c_2$-continuous. Hence $r$ is
a retraction of $\Z^2$ to $X$, with $r|_H$ retracting $H$ to $L_2 \cap X_2$. 
Then $r|_R$ is a retraction of $R$ to $X$. 
It follows from Lemma~\ref{ZnSuffices} that $(X,c_2)$ has the AFPP. 
\end{proof}

\section{Further remarks}
We have continued the work of~\cite{BxConvexity}, exploring relationships
between the convexity of digital images in~$\Z^2$ and the AFPP. In particular,
we have used the result of~\cite{BxConvexity} that convex disks in $\Z^2$
are $c_2$-retracts of digital rectangles and therefore have the AFPP for the
$c_2$-adjacency to show that certain unions of convex disks in $\Z^2$
also have the AFPP.

\section{Acknowledgment}
The suggestions and corrections of an anonymous reviewer are gratefully acknowledged.

\end{document}